\documentclass[11pt, reqno]{amsart}
\usepackage{amsmath, amssymb, amsthm, graphicx,marvosym}
\usepackage[nobysame]{amsrefs}[11pts]

\usepackage{color}

\usepackage{tikz}
\usepackage{caption}
\usepackage{amsmath}
\usepackage{amssymb}
\usepackage{amsthm}
\usepackage{enumerate}
\newtheorem{thm}{Theorem}
\newtheorem{lemma}{Lemma}[section]

\newtheorem{prop}[lemma]{Proposition}
\newtheorem{definition}{Definition}
\newtheorem{remark}{Remark}
\newcommand \nc{\newcommand}
\newcommand{\ben}{\begin{eqnarray}}
\newcommand{\een}{\end{eqnarray}}
\newcommand{\beno}{\begin{eqnarray*}}
\newcommand{\eeno}{\end{eqnarray*}}

\makeatletter \@addtoreset{equation}{section} \makeatother

\nc{\va}{\varphi}
\nc{\ve}{\varepsilon}

\def\ds{\displaystyle}

\newcommand{\R}{\mathbb {R}}

\def\({\left(\begin{array}{cccccc}}
\def\){\end{array}\right)}

\def\bes{\begin{eqnarray}}
\def\ees{\end{eqnarray}}

\title[parabolic PDE with a non-constant H\"older diffusion Coefficient]{On parabolic partial differential equations with H\"older continuous diffusion Coefficients}

\author{Majed Sofiani}
\address[M. Sofiani]{Department of Mathematics, University of Kansas, Lawrence, KS 66045, U.S.A.}
\email{\tt sofiani@ku.edu}

\date{\today}

\begin{document}
\maketitle
%\tableofcontents

\begin{abstract}
We investigate existence and regularity of weak solutions of a 1-dimensional parabolic differential equation with a non-constant H\"older diffusion coefficient and a rough forcing term. Such an equation appears in studying the 1-dimensional Ericksen-Leslie model for nematic liquid crystals where our result applies. The result presented here uses the H\"older continuity of the diffusion coefficient which comes from the physical background and the analysis of the Ericksen-Leslie model. Moreover, the dependence of the H\"older exponent of the solution is explicit on the H\"older exponent of the diffusion coefficient.
    
    % The Ericksen-Leslie model is a coupled system of a hyperbolic equation, describing the alignment of the molecules, and a parabolic equation describing the velocity of the flow.
    % We focus on an equation describing a quantity that carries information from both the alignment and the flow velocity. The equation under study in this paper is set on a general form but we sill discuss and relate to the motivation.
\end{abstract}

% \section{COMMENTS AND QUESTIONS}

% \begin{itemize}
% \item Should I keep the derivation of the equation of $A$ or just discuss the equation of $v$ as a motivation and remove the part about $A.$? In this case, the equation studied here is more general than the motivation.
% \item I am citing the singularity paper. Also, we should cite our new paper as "submitted" and mention it somewhere in the content.
% \item  The reason I think it is better NOT to mention the quantity $A$ here is to leave it for our paper and display the equation for $A$ there. Also, we can discuss there in full details the reason we need to use $A.$
% \end{itemize}

\section{\textbf{Introduction}}
Consider the following Cauchy problem
\begin{align}\label{chy}
    w_t-k(x,t)w_{xx}+\gamma w&=f(x,t)+G(x,t)\quad \text{on}\quad  \Omega_T,\\\label{chyinit}
    w(x,0)&=w_0(x)\quad \text{at}\quad t=0,
\end{align}
where $T>0$ is arbitrary but fixed and  $\Omega_T:=\R\times (0,T].$ We assume
\begin{align}\begin{split}\label{assump}
    k(x,t):&\Omega_T\to \R^+ \in C^{1/2}(\Omega_T),\\
    k_x(x,t)&\in L^{\infty}([0,T],L^2(\R)),
    %k_0(x):=k&(x,0)\in  H^1\cap C^{1,\alpha}(\R);
    \end{split}
\end{align}
%the diffusion coefficient $k(x,t)$ satisfies 
and for some $k_L,k_U>0,$
$$k_L\leq k(x,t)\leq k_U;$$
\begin{align}
    \begin{split}\label{source}
    f(x,t): \overline{\Omega}_T\to \R& \in L^\infty([0,T],L^2(\R));\\
    G(x,t):\overline{\Omega}_T\to \R& \in L^\infty([0,T],C(\R))\cap L^\infty(\Omega_T);\\
    G_x(x,t)&\in L^\infty([0,T],L^1(\R));\\
    \gamma&\in \R;
    \end{split}
\end{align}
and
\begin{align}\label{init}
    w(x,0)=w_0(x)\in H^1(\R).
\end{align}
\begin{remark}\label{G}
The function $G$ has a specific form, namely
\[G(x,t)=\int_{-\infty}^xf^2(z,t)\,dz,\] where $f$ is the function in \eqref{chy}. 
\end{remark}

Equation \eqref{chy} with the specific data \eqref{assump} and \eqref{source} appears in studying the Poiseuille flows of nematic liquid crystals via the Ericksen-Leslie model \cites{ericksen62,leslie68, Les, lin89, CHL20,oseen33,frank58}. This work is particularly inspired by the relatively recent work in \cite{CHL20}. A brief background of equation \eqref{chy} will be presented in \S\ref{1.1}.\\

To state our main result, we need to introduce
\begin{definition}\label{def}
A function $w(x,t)$ is a weak solution to \eqref{chy} and \eqref{chyinit} if the following holds,
\begin{enumerate}
    \item For any $\alpha\in (0,1/4),$
    $$w(x,t),w_x(x,t)\in L^\infty\cap C^\alpha(\Omega_T);$$

    \item for any $\phi\in C_c^1(\Omega_T),$
    \begin{align}
        \int_0^T\int_\R-\phi_t w+(k\phi)_x w_x+\gamma \phi w\,dx\,dt=\int_0^T\int_\R(f+G)\phi\,dx\,dt;
    \end{align}

    \item The initial data is satisfied in the following sense, $$w(x,t)\to w_0(x) \quad \text{point-wise as} t\to 0^+.$$
\end{enumerate}
\end{definition}

\begin{thm}
For any fixed $T>0,$ under the assumptions \eqref{assump}-\eqref{init} the Cauchy problem \eqref{chy} and \eqref{chyinit} has a weak solution in the sense of {Definition \ref{def}} over the domain $\Omega_T.$
\end{thm}

% The goal of this paper is to formulate a weak solution of the Cauchy problem \eqref{chy} and \eqref{chyinit} using a type of Duhamel principal, estimate the $L^2$,$L^\infty$ and $C^\alpha$ norms of such solution and its derivative with respect to $x$ for some $0<\alpha<1/4,$ and show the existence of such solution satisfying the weak formulation defined above.\\

\subsection{\textbf{{Physical background and motivation}}}\label{1.1}

As mentioned above, this work is motivated by a study of Ericksen-Leslie model for Poiseulle flow of nematic liquid crystals. The state variables for nematics are the velocity vector $\bf u$ and the director field $\bf n.$
In the case where the flow moves in the $z$-direction and is uniform in $x,y$ and the molecules lie on the $x-z$ plane, we write the velocity field as $\mathbf{u}(x,t)=(0,0,u(x,t))^T$ and the director field as $\mathbf{n}(x,t)=(\sin\theta(x,t),0,\cos\theta(x,t))^T,$ where $\theta$ is the angle measured from the $z$ axis. The system reduces to
\begin{align}\label{ugh}\begin{split}
    \ds u_t&=\left(g(\theta)u_x+h(\theta)\theta_t\right)_x,\\
\theta_{tt}+\gamma_1\theta_t&=c(\theta)\big(c(\theta)\theta_x\big)_x-h(\theta)u_x,
\end{split}
\end{align}
where the functions $c,$ $g,$ and $h$ are explicitly given in the literature
\begin{align}\label{fgh}\begin{split}
 g(\theta):=&\alpha_1\sin^2\theta\cos^2\theta+\frac{\alpha_5-\alpha_2}{2}\sin^2\theta+\frac{\alpha_3+\alpha_6}{2}\cos^2\theta+\frac{\alpha_4}{2},\\
 %f(\theta)\,\equiv&c^2(\theta):=K_1\cos^2\theta+K_3\sin^2\theta,\\
  h(\theta):=&\alpha_3\cos^2\theta-\alpha_2\sin^2\theta=\frac{\gamma_1+\gamma_2\cos(2\theta) }{2},\\
  c^2(\theta):=&K_1\cos^2\theta+K_3\sin^2\theta.
 \end{split}
 \end{align}
 In \eqref{fgh}, $K_1, K_3$ are Frank's constants for splay and bending deformation of the material and $\alpha_j's$ with ($\gamma_1=\alpha_3-\alpha_2,\, \gamma_2=\alpha_6-\alpha_5$) are the Leslie dynamics parameters. See \cite{CHL20} for a detailed discussion and derivation of the Ericksen-Leslie system for Poiseulle flows.\\

In \cite{CHL20}, a special case ($g=h=1$ and $\gamma_1=2$) was investigated and a crucial quantity $J:=u_x+\theta_t$ was realized. It turns out that $J$ has better regularity than $\theta_t$ and $u_x$ individually. That is, each of $\theta_t$ and $u_x$ might blow up in finite time but $J$ will always be bounded and in fact H\"older continuous. The system is written as
\begin{align}\label{uj}
    \ds u_t&=\left(u_x+\theta_t\right)_x,\\
\theta_{tt}+2\theta_t&=c(\theta)\big(c(\theta)\theta_x\big)_x-u_x.
\end{align}
In terms of $J,$
\begin{align}\label{ujj}
  u_t&=J_x,\\\label{thetaj}
 \theta_{tt}+\theta_t&=c(\theta)(c(\theta)\theta_x)_x-J.
\end{align}
To close the system, using the definition of $J$ and \eqref{ujj}, it is easy to see that
\[J_t=J_{xx}+\theta_{tt}.\]
Using the wave equation we obtain an equation for $J$
\begin{align}\label{J}
   J_t=J_{xx}+c(\theta)(c(\theta)\theta_x)_x-u_x-\theta_{t}. 
\end{align}
This parabolic PDE \eqref{J} has a corresponding parabolic PDE for the general case \eqref{ugh}, with non-constant $g$ and $h,$ satisfied by a quantity $A$ expressed as follows
$$A(x,t):=\int_{-\infty}^xJ(z,t)\,dz$$
where $J$ in this case is slightly modified,
$$J:=u_x+\frac{h}{g}\theta_t.$$
For full context of the derivation of the quantity and its equation we refer the readers to \cite{CLS}.\\

The parabolic PDE satisfied by $A$ has a similar form as equation \eqref{chy} considered in this paper. We will investigate this equation in a general setup that is closely related to the situation for the system \eqref{ugh}. The results presented here are crucial for the proof of the global existence of weak solution for the full Ericksen-Leslie  model  for Poiseulle flows of nematic liquid crystals in \cite{CHL20}. That is mainly because the major difference is the appearance of the function $g(\theta)$ as a diffusion coefficient in the equations for the quantities $v$ and $A$ defined above.\\

% The quantity $A$ satisfies some parabolic differential equation with some H\"older diffusion coefficient $g(\theta).$
% The function $g(\theta)$ is given explicitly in the literature and it is bounded, smooth, and strictly positive. 
% In general, the solution $\theta$ is H\"older continuous with exponent $1/2$ so the diffusion coefficient is only $C^{1/2}.$\\

% It is worth mentioning that the H\"older continuity of $\theta(x,t)$ is deduced from the Sobolev embedding of $H^1.$\\

\subsection{\textbf{A relevant work from \cite{Fri} and Duhamel formula.}}\label{1.2}
 In this subsection we recall a result from \cite{Fri} on the existence of classical solutions to parabolic PDEs. The framework therein is relevant and is our starting point for the analysis in this work. The results in \cite{Fri} are shown for the n-dimensional space and any H\"older exponent $0<\beta<1.$ We focus on the relevant situation when $n=1$ and $\beta=1/2.$ 
 \\
 
 Let $k(x,t)\in C^{\beta}$ with respect to both variables on $\Omega_T.$
 Consider, the following homogeneous equation
 \begin{align}\label{homog}
   w_t-k(x,t)w_{xx}+\gamma w=0.  
 \end{align}
% By freezing the coefficient $k(x,t)|_{(\xi,\tau)}$ and ignoring the last term in \eqref{homog}, we obtain the heat kernel
% $H^{\xi,\tau}(x-\xi,t-\tau)$ satisfying
% and the explicit formula given by,
Note that, for fixed $(\xi,\tau),$ the function
\begin{align}
    H^{\xi,\tau}(x-\xi,t-\tau)=\frac{1}{2\sqrt{k(\xi,\tau)}\sqrt{t-\tau}}e^{-\frac{(x-\xi)^2}{4k(\xi,\tau)(t-\tau)}}
\end{align} 
is the heat kernel with frozen coefficient $k(\xi,\tau),$ that is,
\[H_t^{\xi,\tau}-k(\xi,\tau)H_{xx}^{\xi,\tau}=0.\]
We can write the fundamental solution $\Gamma$ of \eqref{homog} implicitly as follows

\begin{align}\label{GGGamma}
\Gamma(x,t,\xi,\tau)&=H^{\xi,\tau}(x-\xi,t-\tau)
+\int_\tau^t\int_\mathbb{R}H^{y,s}(x-y,t-s)\Phi(y,s;\xi,\tau)\,dy\,ds,
\end{align}
where the function $\Phi$ is determined by the condition that $$\Gamma_t-k(x,t)\Gamma_{xx}+\gamma\Gamma=0.$$
In \cite{Fri}, the author showed such function $\Phi$ exists and satisfies
    \begin{align}\label{phiest}
|\Phi(y,s;\xi,\tau)|\leq \frac{C}{(s-\tau)^{\frac{3-\beta}{2}}}e^{\frac{-d(y-\xi)^2}{4(s-\tau)}}=\frac{C}{(s-\tau)^{\frac{5}{4}}}e^{\frac{-d(y-\xi)^2}{4(s-\tau)}},
\end{align}
where $C$ and $d$ are constants depending only on $k.$ We will use $C$ to represent other constants that are not necessarily the same. In addition, the author showed that if $q(x,t)$ is H\"older continuous on $\Omega_T$ and $\psi_0(x)$ is continuous on $\R$, then the solution of the following Cauchy problem
\[-u_t+k(x,t)u_{xx}-\gamma u=q(x,t),\]
\[u(x,0)=\psi_0(x),\]
 is given by the Duhamel formula
\begin{align}\label{sol}
     u(x,t)=\int_\R\Gamma(x,t,\xi,0)\psi_0(\xi)\,d\xi+\int_0^t\int_\R\Gamma(x,t;\xi.\tau)q(\xi,\tau)\,d\xi\,d\tau,
 \end{align}
and it is a classical solution.\\

% where
% \begin{align}\label{GGGamma}
% \Gamma(x,t,\xi,\tau)&=H^{\xi,\tau}(x-\xi,t-\tau)
% +\int_\tau^t\int_\mathbb{R}H^{y,s}(x-y,t-s)\Phi(y,s;\xi,\tau)\,dy\,ds,
% \end{align}
% The function $\Phi$ is determined by the condition that $$-\Gamma_t+g(\theta)\Gamma_{xx}+\gamma\Gamma=0,$$ and it can be shown that it satisfies
%     \begin{align}\label{phiest}
% |\Phi(y,s;\xi,\tau)|\leq \frac{const}{(s-\tau)^{5/4}}e^{\frac{-d(y-\xi)^2}{4(s-\tau)}},
% \end{align}
% where $d$ is a constant depending on $g.$ Moreover we have the following estimates,
In addition, the following estimates were shown therein,
\begin{align}\label{Gammaest}
    |\Gamma(x,t;\xi,\tau)|\leq \frac{C}{\sqrt{t-\tau}}e^{-\frac{d(x-\xi)^2}{4(t-\tau)}}\approx H(x-\xi,t-\tau),
\end{align}
\begin{align}\label{Gamma_xest}
    |\Gamma_x(x,t;\xi,\tau)|\leq \frac{C}{{t-\tau}}e^{-\frac{d(x-\xi)^2}{4(t-\tau)}}\approx\frac{1}{\sqrt{t-\tau}}H(x-\xi,t-\tau),
\end{align}
where $d$ is a constant depending on $k$ and $\approx$ means equality up to a factor depending on $k(\xi,\tau).$

\subsection{\textbf{The main idea of this work}}
 For the application in our work \cite{CLS}, we encounter a rougher, particularly not H\"older continuous, data for the non-homogeneous term $q(x,t).$ For this reason, the result in \cite{Fri} can not be applied. Nevertheless, an extension of the framework in \cite{Fri} is needed in the sense of proving the existence of a weak solution in the case when a rough source data imposed to the PDE.
 In addition, H\"older estimates for the weak solution and its $x$-derivative are needed for the work in \cite{CLS}. 
 We estimate the $L^\infty, L^2$ and the $C^\alpha$ norms of the solution and its derivative with respect to $x.$ We use the mollifying technique to prove the existence of weak solution. Precisely, we mollify the data $f$ and $G$ and then pass to the limit with respect to the $L^\infty$ norm. Briefly, if we mollify $f$ and $G$ we can write the solution in a classical sense following \cite{Fri}. Then, using the $L^\infty$ estimates of $u$ and $u_x$ we can take the limit of the weak formulation with respect to the mollifier parameter. It is worth mentioning that H\"older continuity of solutions to parabolic PDE's has been extensively studied and many results have been shown, for example \cites{krylov-saf,lihewang} and the references there. However, since the equation here is motivated by a concrete model as discussed before,
 the result here assumes H\"older continuity of the coefficient and gives more explicit formula for the solution and also shows the dependence on the H\"older exponent of the coefficient.\\
 
Following the discussion in \S\ref{1.2}, we can represent the solution to \eqref{chy} and \eqref{chyinit} as follows,
\begin{align}\label{soll}
w(x,t)=\int_\R\Gamma(x,t,\xi,0)w_0(\xi)\,d\xi+\int_0^t\int_\R\Gamma(x,t;\xi.\tau)[G(\xi,\tau)+f(\xi,\tau)]\,d\xi\,d\tau.
 \end{align}
We will show that this representation of the solution indeed satisfies the weak formulation in Definition \ref{def}. We will consider the three terms of the expression \eqref{soll} separately and denote each term as follows,
\begin{align*}
W(x,t):=\int_\R\Gamma(x,t,\xi,0)w_0(\xi)\,d\xi,
\end{align*}
\begin{align*}
W_G(x,t):=\int_0^t\int_\R\Gamma(x,t;\xi.\tau)G(\xi,\tau)\,d\xi\,d\tau,
\end{align*}
and
\begin{align*}
W_f(x,t):=\int_0^t\int_\R\Gamma(x,t;\xi.\tau)f(\xi,\tau)\,d\xi\,d\tau.
\end{align*}
Analogously, denote the $x$ derivative of the expressions as 
\begin{align*}
W_x(x,t):=\int_\R\Gamma_x(x,t,\xi,0)w_0(\xi)\,d\xi.
\end{align*}
\begin{align*}
    W_{G,x}(x,t):=\int_0^t\int_\R\Gamma_x(x,t;\xi.\tau)G(\xi,\tau)\,d\xi\,d\tau,
\end{align*}
and
\begin{align*}
W_{f,x}(x,t):=\int_0^t\int_\R\Gamma_x(x,t;\xi.\tau)f(\xi,\tau)\,d\xi\,d\tau.
\end{align*}
The key estimates on the solution $w$ are presented in the next propositions.

\begin{prop}\label{prop1}
Fix $T>0.$
%$$f(x,t)\in L^\infty([0,T],L^2(\R)),\quad \text{and}\quad G(x,t)\in L^\infty([0,T],C(\R))\cap L^\infty([0,T],\R).$$
Under the assumptions \eqref{assump}- \eqref{init}, the functions $W_G, W_f, W_{G,x},$ and $W_{f,x}$ belong to the following space,
\[L^\infty\cap C^\alpha(\bar \Omega_T),\] for any $0<\alpha<1/4.$ 
% and
% \[\|M_{\{f\}}(x,t)\|_{(L^2)(\overline \Omega_\delta)}\,,\,\|M_{\{f,x\}}(x,t)\|_{(L^2)(\overline \Omega_\delta)}\leq K,\]
% \[\|M_{\{G,x\}}(x,t)\|_{(L^2)(\overline \Omega_\delta)}\leq K,\]
Moreover, for any compact set $F\subset \Omega_T,$
% with the assumption in \eqref{assump}, $$k_0:=k(x,0)\in  H^1\cap C^{1,\alpha}(\R)$$ and \eqref{init},
one has
\[\max\big\{\|W\|_{(L^\infty\cap C^\alpha)(F)},\|W_x\|_{(L^\infty\cap C^\alpha)(F)}\big\}<\infty,\]
and
\[W(x,t),W_x(x,t)\to w_0(x), w'_0(x) \mbox{ point-wise as } t\to 0^+.\] 
% If, in addition, $k'_0(x):=\frac{d}{dx}k(x,0)\in C^0(\R),$ then for any $1\leq p< 4$ we have
% \[\|W_x(x,s)-w'_0(x)\|_{L^p([0,t],L^\infty(\R))}\to 0 \mbox{ as } t\to 0^+.\]
\end{prop}
As a consequence of the first assertion in  Proposition \ref{prop1}, we have that for any constant $K>0,$ there exists $\delta>0,$ depending on $G,f$ and $k,$ sufficiently small such that
\[\|W_{\{\square\}}(x,t)\|_{(L^\infty\cap C^\alpha)(\overline \Omega_\delta)}\leq K,\]
for $\{\square\}:=\{G\},\{f\},\{G,x\}$ and $\{f,x\}.$\\
% % \begin{remark}
% % Due to the lack of decay of $G,$ the $L^2$ estimate for $M_{G}(x,t)$ can not be obtained.
% % \end{remark}

Due to the lack of decay of the source term $G,$ we can only obtain $L^2$ decay of the terms $W_f,W_{f,x}$ and $W_{G,x}.$ More precisely,
\begin{prop}\label{prop}
 For any constant $K>0,$ there exists $\delta>0,$ depending on $G,f$ and $k,$ sufficiently small such that
\[\|W_{\{\circ\}}(x,t)\|_{L^2(\overline \Omega_\delta)}\leq K,\]
for $\{\circ\}:=\{f\},\{f,x\},$ and \ $\{G,x\}.$
\end{prop}

The rest of the paper is organised as follows. In {Section 2}, we present a proof of {Proposition \ref{prop1}}. In {Section 3}, we prove the global existence of the weak solution ({Theorem 1}). Finally, we prove {Proposition \ref{prop}} in {Section 4}.

\section{\textbf{Proof of Proposition \ref{prop1}}}
To avoid repetitions, we will show the estimates for $W_{G,x}$ and $W_{f,x}$. The other terms have no derivatives on $\Gamma$ so they have even weaker integrable singularities and can be estimated in a similar fashion.
\subsection{\textbf{The $L^\infty$ estimates}}
In this subsection, we consider the $L^\infty$ estimates for $W_{G,x}$ and $W_{f,x}.$\\

For the term
\begin{align}
W_{G,x}:=\int_0^t\int_\R\Gamma_x(x,t;\xi.\tau)G(\xi,\tau)\,d\xi\,d\tau,
\end{align}
we use \eqref{Gamma_xest} to get,
\begin{align*}
    |W_{G,x}|\lesssim \|G\|_{L^\infty(\Omega_T)}\int_0^t\int_\R\frac{1}{{t-\tau}}e^{-\frac{d(x-\xi)^2}{4(t-\tau)}}\,d\xi\,d\tau\lesssim T^{1/2}\|G\|_{L^\infty(\Omega_T)},
\end{align*}
which yields,
\begin{align}\label{infityM_Gx}
 \|W_{G,x}\|_{L^\infty(\Omega_T)}\lesssim T^{1/2}\|G\|_{L^\infty(\Omega_T)}.   
\end{align}
Similarly, for the term
\begin{align}
W_{f,x}:=\int_0^t\int_\R\Gamma_x(x,t;\xi.\tau)f(\xi,\tau)\,d\xi\,d\tau,
\end{align}
we use \eqref{Gamma_xest} and H\"older inequality,
\begin{align*}
 |W_{f,x}|\lesssim& \bigg[\int_0^t\int_\R\frac{1}{{|t-\tau|}^{2-2r}}e^{-\frac{d(x-\xi)^2}{2(t-\tau)}}\,d\xi\,d\tau\bigg]^{1/2} \bigg[\int_0^t\int_\R\frac{1}{{|t-\tau|}^{2r}}f^2\,d\xi\,d\tau\bigg]^{1/2}\\
 &\lesssim \bigg[\int_0^t\frac{1}{{|t-\tau|}^{\frac{3}{2}-2r}}\,d\tau\bigg]^{1/2}\bigg[\int_0^t\frac{1}{{|t-\tau|}^{2r}}\,d\tau\bigg]^{1/2}\|f\|_{L^\infty((0,T),L^2(\R))}.
\end{align*}
For $r=\frac{3}{8}$ we obtain:
\[|W_{f,x}|\lesssim T^{1/4}\|f\|_{L^\infty((0,T),L^2(\R))},\]
which yields,
\begin{align}\label{infityM_fx}
\|W_{f,x}\|_{L^\infty(\Omega_T)}\lesssim T^{1/4}\|f\|_{L^\infty((0,T),L^2(\R))}. 
\end{align}

\subsection{\textbf{The $C^\alpha$ estimates}}

Recall,
\begin{align}\nonumber
    &\int_0^t\int_{\R}\Gamma_x(x,t;\xi,\tau)\square(\xi,\tau)\,d\xi\,d\tau\\\label{nophi}
    &=\int_0^t\int_{\R}H_x^{\xi,\tau}(x-\xi,t-\tau)\square(\xi,\tau)\,d\xi\,d\tau\\\label{wphi}
    &+\int_0^t\int_{\R}\bigg[\int_\tau^t\int_{\R}H_x^{y,s}(x-y,t-s)\Phi(y,s;\xi,\tau)\,dy\,ds\bigg]\square(\xi,\tau)\,d\xi\,d\tau.
\end{align}

We call \eqref{nophi} the 'principal term' and \eqref{wphi} the 'lower order term.' We need to estimate the $C^\alpha$ norm in $x$ for both terms.\\

Starting with the principal term, we write $H_x^{\xi,\tau}$ explicitly:
\begin{align*}
    H_x^{\xi,\tau}(x-\xi,t-\tau)=-\frac{(x-\xi)}{4\sqrt{\pi}k(\xi,\tau)^{3/2}(t-\tau)^{3/2}}e^{-\frac{(x-\xi)^2}{4k(\xi,\tau)(t-\tau)}}.
\end{align*}
Without loss of generality, fix $x_1<x_2$ and consider the intervals $$B_1:=|x_1-\xi|<|x_2-\xi| \mbox{ and } B_2:=|x_2-\xi|<|x_1-\xi|.$$
On $B_1$ we use
\begin{align*}
    &\big|H_x(x_2-\xi,t-\tau)-H_x(x_1-\xi,t-\tau)\big|\\
    &\leq \frac{1}{k^{3/2}_L}\bigg|\frac{x_1-x_2}{(t-\tau)^{3/2}}e^{-\frac{(x_2-\xi)^2}{4k_U(t-\tau)}}\bigg|+\frac{1}{k^{3/2}_L}\bigg|\frac{x_1-\xi}{(t-\tau)^{3/2}}\bigg[e^{-\frac{(x_1-\xi)^2}{4k(\xi,\tau)(t-\tau)}}-e^{-\frac{(x_2-\xi)^2}{4k(\xi,\tau)(t-\tau)}}\bigg]\bigg|,
\end{align*}
while on $B_2$ we use
\begin{align*}
    &\big|H_x(x_2-\xi,t-\tau)-H_x(x_1-\xi,t-\tau)\big|\\
    &\leq \frac{1}{k^{3/2}_L}\bigg|\frac{x_1-x_2}{(t-\tau)^{3/2}}e^{-\frac{(x_1-\xi)^2}{4k_U(t-\tau)}}\bigg|+\frac{1}{k^{3/2}_L}\bigg|\frac{x_2-\xi}{(t-\tau)^{3/2}}\bigg[e^{-\frac{(x_2-\xi)^2}{4k(\xi,\tau)(t-\tau)}}-e^{-\frac{(x_1-\xi)^2}{4k(\xi,\tau)(t-\tau)}}\bigg]\bigg|.
\end{align*}
We will consider the case on $B_1.$ The estimates on $B_2$ follow similarly.\\

Dividing the inequality by $|x_2-x_1|^\alpha$ we get

\begin{align}\nonumber
    &\frac{\big|H_x(x_2-\xi,t-\tau)-H_x(x_1-\xi,t-\tau)\big|}{|x_2-x_1|^\alpha}\\\label{betax}
    &\lesssim \frac{|x_2-x_1|^{1-\alpha}}{|t-\tau|^{3/2}}e^{-\frac{(x_2-\xi)^2}{4k_U(t-\tau)}}+\frac{|x_1-\xi|}{|t-\tau|^{3/2}|x_2-x_1|^\alpha}\bigg|e^{-\frac{(x_1-\xi)^2}{4k(\xi,\tau)(t-\tau)}}-e^{-\frac{(x_2-\xi)^2}{4k(\xi,\tau)(t-\tau)}}\bigg|.
\end{align}
Now for the principal part \eqref{nophi} we have two terms produced by \eqref{betax}. We start with the first term in \eqref{betax}:

\begin{align*}
    &\int_0^t\int_{B_1}\frac{|x_2-x_1|^{1-\alpha}}{|t-\tau|^{3/2}}e^{-\frac{(x_2-\xi)^2}{4k_U(t-\tau)}}|\square(\xi,\tau)|\,d\xi\,d\tau\\
    &\leq
    \underbrace{2\int_0^t\int_{B_1}\frac{|x_2-\xi|^{1-\alpha}}{|t-\tau|^{3/2}}e^{-\frac{(x_2-\xi)^2}{4k_U(t-\tau)}}|\square(\xi,\tau)|\,d\xi\,d\tau}_{I_\square}.
\end{align*}

For the second term in \eqref{betax} we have,

\begin{align}\nonumber
   &\int_0^t\int_{B_1}\frac{|x_1-\xi|}{|t-\tau|^{3/2}|x_2-x_1|^\alpha}\bigg|e^{-\frac{(x_1-\xi)^2}{4k(t-\tau)}}-e^{-\frac{(x_2-\xi)^2}{4k(t-\tau)}}\bigg||\square(\xi,\tau)|\,d\xi\,d\tau\\\nonumber
   &=\int_0^t\int_{B_1}\frac{|x_1-\xi|}{|t-\tau|^{3/2}|x_2-x_1|^\alpha}e^{-\frac{(x_1-\xi)^2}{4k(t-\tau)}}\bigg|1-e^{\frac{(x_1-\xi)^2}{4k(t-\tau)}-\frac{(x_2-\xi)^2}{4k(t-\tau)}}\bigg||\square(\xi,\tau)|\,d\xi\,d\tau\\\label{2-20}
   &=\int_0^t\int_{B_1}\frac{|x_1-\xi|}{|t-\tau|^{3/2}|x_2-x_1|^\alpha}e^{-\frac{(x_1-\xi)^2}{4k(t-\tau)}}\bigg|1-e^{-\frac{(x_2-x_1)(x_2+x_1-2\xi)}{4k(t-\tau)}}\bigg||\square(\xi,\tau)|\,d\xi\,d\tau.
\end{align}
Note that for $\xi\in B_1$ we have
\[\frac{(x_2-x_1)(x_2+x_1-2\xi)}{4k(\xi,\tau)(t-\tau)}>0.\] 
The following lemma will be needed,
\begin{lemma}\label{lemma-e}
For all $a\geq 0$ and $0<\lambda\leq 1,$ one has
\[1-e^{-a}\leq \frac{1}{\lambda} a^{\lambda}.\] 
\end{lemma}
 \begin{proof}
 The case $\lambda=1$ can be shown straightforwardly by the mean value theorem. For $0<\lambda<1,$ define $F(a)=\frac{1}{\lambda} a^{\lambda}+e^{-a}-1.$ This is a continuous function on $[0,\infty)$ for any $\lambda\in (0,1).$ Since $F(0)=0$ and $\lim_{a\to\infty}F(a)=\infty.$ It suffices to show that the values of $F$ at local extrema are all non-negative.\\
 
If $F$ has no extrema then it is monotone and we are done. Otherwise, suppose $\tilde a > 0$ is a local extremum of $F,$ that is $F'(\tilde a)=0,$ which gives
 \[{\tilde a}^{\lambda-1}=e^{-{\tilde a}}.\]
Now we have,
 \[F(\tilde a)=\frac{1}{\lambda} \tilde{a}^{\lambda}+{\tilde a}^{\lambda-1}-1:=\tilde F.\]
We claim that $\tilde F>0$ for all $\tilde a>0$ and hence if $\tilde a$ is a local minimum then $F(\tilde a)$ is  still positive. To show that we notice,
 \[\tilde F\to\infty,\] as $\tilde a\to 0^+,$ and 
 \[\tilde F\to \infty,\] as $\tilde a\to \infty.$ Moreover, there is only one extremum point $a^*$ that is
 \[{a^{*}}^{\lambda-1}+(\lambda-1){a^*}^{\lambda-2}=0,\]
%  \[\bar a^{\lambda-2}(\bar a+\lambda-1)=0.\]
 which gives
 \[a^*=1-\lambda.\] It must be a global minimum and finally
 \[\tilde F(1-\lambda)=\frac{1}{\lambda}(1-\lambda)^\lambda+(1-\lambda)^{\lambda-1}-1>0\] if and only if
 \[\frac{1-\lambda}{\lambda}+1-(1-\lambda)^{1-\lambda}>0,\]
 if and only if
 \[1-\lambda(1-\lambda)^{1-\lambda}>0,\] which is clearly true since $0<\lambda<1.$
 \end{proof}
Now using Lemma \ref{lemma-e} for \eqref{2-20} with $\lambda=\frac{1}{2}-\epsilon,$ for some $\epsilon$ to be determined later, and 
\[a=\frac{(x_2-x_1)(x_2+x_1-2\xi)}{4k(\xi,\tau)(t-\tau)},\] we have
\begin{align}\label{II,III}\nonumber
    &\int_0^t\int_{B_1}\frac{|x_1-\xi|}{|t-\tau|^{3/2}|x_2-x_1|^\alpha}e^{-\frac{(x_1-\xi)^2}{4k(t-\tau)}}\bigg|1-e^{-\frac{(x_2-x_1)(x_2+x_1-2\xi)}{4k(t-\tau)}}\bigg||\square(\xi,\tau)|\,d\xi\,d\tau\\\nonumber
    &\lesssim \int_0^t\int_{B_1}\frac{|x_1-\xi|}{|t-\tau|^{3/2}|x_2-x_1|^\alpha}e^{-\frac{(x_1-\xi)^2}{4k_U(t-\tau)}}\bigg|\frac{(x_2-x_1)(x_2+x_1-2\xi)}{4(t-\tau)}\bigg|^{\frac{1}{2}-\epsilon}|\square(\xi,\tau)|\,d\xi\,d\tau\\\nonumber
    &\leq\int_0^t\int_{B_1}\frac{|x_1-\xi|}{|t-\tau|^{3/2}|x_2-x_1|^\alpha}e^{-\frac{(x_1-\xi)^2}{4k_U(t-\tau)}}\frac{(x_2-x_1)^{1-2\epsilon}+[2(x_2-x_1)|x_1-\xi|]^{\frac{1}{2}-\epsilon}}{[4(t-\tau)]^{\frac{1}{2}-\epsilon}}|\square(\xi,\tau)|\,d\xi\,d\tau\\\nonumber
    &\lesssim \int_0^t\int_{B_1}\frac{|x_1-\xi|}{|t-\tau|^{2-\epsilon}}e^{-\frac{(x_1-\xi)^2}{4k_U(t-\tau)}}{(x_2-x_1)^{1-2\epsilon-\alpha}}|\square(\xi,\tau)|\,d\xi\,d\tau\\
    &+\int_0^t\int_{B_1}\frac{|x_1-\xi|^{\frac{3}{2}-\epsilon}}{|t-\tau|^{2-\epsilon}}e^{-\frac{(x_1-\xi)^2}{4k_U(t-\tau)}}{(x_2-x_1)^{\frac{1}{2}-\epsilon-\alpha}}|\square(\xi,\tau)|\,d\xi\,d\tau\lesssim II_\square+III_\square.
\end{align}    
%     \begin{align}\nonumber
%     \lesssim \int_0^t\int_{B_1}\frac{|x_1-\xi|}{|t-\tau|^{2-\epsilon}}e^{-\frac{(x_1-\xi)^2}{4k_U(t-\tau)}}&|\square(\xi,\tau)|\,d\xi\,d\tau\\
%     &+\int_0^t\int_{B_1}\frac{|x_1-\xi|^{\frac{3}{2}-\epsilon}}{|t-\tau|^{2-\epsilon}}e^{-\frac{(x_1-\xi)^2}{4k_U(t-\tau)}}|\square(\xi,\tau)|\,d\xi\,d\tau.
% \end{align}
where $II_\square$ and $III_\square$ are given by
\begin{align*}
    II_\square:= \int_0^t\int_{B_1}\frac{|x_1-\xi|}{|t-\tau|^{2-\epsilon}}e^{-\frac{(x_1-\xi)^2}{4k_U(t-\tau)}}&|\square(\xi,\tau)|\,d\xi\,d\tau,
\end{align*}
\begin{align*}
    III_\square:=\int_0^t\int_{B_1}\frac{|x_1-\xi|^{\frac{3}{2}-\epsilon}}{|t-\tau|^{2-\epsilon}}e^{-\frac{(x_1-\xi)^2}{4k_U(t-\tau)}}|\square(\xi,\tau)|\,d\xi\,d\tau.
\end{align*}
The last inequality in \eqref{II,III} follows from the requirements $1-2\epsilon-\alpha>0$ and $\frac{1}{2}-\epsilon-\alpha>0.$ Note that the second requirement implies the first. Now, we choose
\begin{align}\label{eps-range}
 \epsilon<\frac{1}{2}-\alpha.   
\end{align}

% Using the mean value theorem: there exists $z\in (x_1,x_2)$ such that
% \begin{align}\nonumber
%     \bigg|e^{-g_L\frac{(x_1-\xi)^2}{4(t-\tau)}}-e^{-g_L\frac{(x_2-\xi)^2}{4(t-\tau)}}\bigg|&=\big|-g_U\frac{z-\xi}{2(t-\tau)}e^{-g_L\frac{(z-\xi)^2}{4(t-\tau)}}\big|(x_2-x_1)\\
%     &\leq\label{mean}
%     g_U\frac{|x_1-\xi|}{2(t-\tau)}e^{-g_L\frac{(x_1-\xi)^2}{4(t-\tau)}}(x_2-x_1)
% \end{align}
% where the inequality holds true in some unbounded interval of $\xi,$ say $\mathcal{D},$ In another unbounded interval, say $\mathcal{D}'$ we have 

% \begin{align}\nonumber
%     \bigg|e^{-g_L\frac{(x_1-\xi)^2}{4(t-\tau)}}-e^{-g_L\frac{(x_2-\xi)^2}{4(t-\tau)}}\bigg|&=\big|-g_U\frac{z-\xi}{2(t-\tau)}e^{-g_L\frac{(z-\xi)^2}{4(t-\tau)}}\big|(x_2-x_1)\\
%     &\leq
%     g_U\frac{|x_2-\xi|}{2(t-\tau)}e^{-g_L\frac{(x_2-\xi)^2}{4(t-\tau)}}(x_2-x_1).
% \end{align}
 
%  In between the two unbounded intervals there is a bounded interval where either of the estimates above is true up to a constant factor. We will just show the estimates in $\mathcal{D}.$ The calculations in the other intervals are similar.

% Over $\mathcal{D},$ We estimate the integral \eqref{2-20} as
% \begin{align*}
%   \int_0^t\int_{\mathcal{D}}g_U&\frac{|x_1-\xi|}{|t-\tau|^{3/2}|x_2-x_1|^\alpha}\bigg|e^{-g_L\frac{(x_1-\xi)^2}{4(t-\tau)}}-e^{-g_L\frac{(x_2-\xi)^2}{4(t-\tau)}}\bigg||\square(\xi,\tau)|\,d\xi\,d\tau.\\
%   &\lesssim \underbrace{\int_0^t\int_{\R}  \frac{|x_1-\xi|^2|x_2-x_1|^{1-\alpha}}{2(t-\tau)^{3/2}(t-\tau)}e^{-g_L\frac{(x_1-\xi)^2}{4(t-\tau)}}|\square(\xi,\tau)|\,d\xi\,d\tau}_{III_\square}.
% \end{align*}

Finally, we have the "lower order term" \eqref{wphi}.

\begin{align}
IV_\square:=&\int_0^t\int_\R\int_\tau^t\int_{B_1}\frac{\big|H_x(x_2-y,t-s)-H_x(x_1-y,t-s)\big|}{|x_2-x_1|^\alpha}|\Phi(y,s;\xi,\tau)|\,dy\,ds\,|\square(\xi,\tau)|\,d\xi\,d\tau\\
&=\int_0^t\int_\R\int_\tau^t\int_{B_1}\|H_x(x-y,t-s)\|_{C^\alpha_x}|\Phi(y,s;\xi,\tau)|\,dy\,ds\,|\square(\xi,\tau)|\,d\xi\,d\tau,
\end{align}
where, for short, we write
\[\|H_x(x-y,t-s)\|_{C^\alpha_x}:=\frac{\big|H_x(x_2-y,t-s)-H_x(x_1-y,t-s)\big|}{|x_2-x_1|^\alpha}.\]
Therefore, to estimate the $C^\alpha_x$ norm of $M_\square(x,t)$ it suffices to bound the integrals

\[I_\square, II_\square, III_\square\,\,\,\text{and}\,\,\, IV_\square,\] 
for $\square=\{G,x\}$ and $\{f,x\}.$\\

\begin{remark}
Notice that if $\alpha>0$ is the H\"older exponent of $H_x(x,t)$ with respect to $x$ then $\alpha/2$ is the H\"older exponent of $H_x(x,t)$ with respect to $t.$ This is due to the scaling difference of the $x$ and $t$ variables of the heat kernel $H,$ that is 
\[(H_x)_t=(H_x)_{xx}.\]
\end{remark}

To this end, recall

% \underbrace{\int_0^t\int_{R}g_U\frac{|x_2-\xi|^{1-\alpha}}{|t-\tau|^{3/2}}e^{-g_L\frac{(x_2-\xi)^2}{4(t-\tau)}}|\square(\xi,\tau)|\,d\xi\,d\tau}_{I_\square}
%     +\underbrace{\int_0^t\int_{R}g_U\frac{|x_1-\xi|^{1-\alpha}}{|t-\tau|^{3/2}}e^{-g_L\frac{(x_2-\xi)^2}{4(t-\tau)}}|\square(\xi,\tau)|\,d\xi\,d\tau}_{II_\square}
\begin{align*}
    I_\square&=\int_0^t\int_{B_1}\frac{|x_2-\xi|^{1-\alpha}}{|t-\tau|^{3/2}}e^{-\frac{(x_2-\xi)^2}{4k_U(t-\tau)}}|\square(\xi,\tau)|\,d\xi\,d\tau,\\
    II_\square&=\int_0^t\int_{B_1}\frac{|x_1-\xi|}{|t-\tau|^{2-\epsilon}}e^{-\frac{(x_1-\xi)^2}{4k_U(t-\tau)}}|\square(\xi,\tau)|\,d\xi\,d\tau,\\
    III_\square&=\int_0^t\int_{B_1}\frac{|x_1-\xi|^{\frac{3}{2}-\epsilon}}{|t-\tau|^{2-\epsilon}}e^{-\frac{(x_1-\xi)^2}{4k_U(t-\tau)}}|\square(\xi,\tau)|\,d\xi\,d\tau,
\end{align*}
and
\begin{align*}
IV_\square:=\int_0^t\int_\R\int_\tau^t\int_{B_1}\|H_x(x-y,t-s)\|_{C^\alpha_x}|\Phi(y,s;\xi,\tau)|\,dy\,ds\,|\square(\xi,\tau)|\,d\xi\,d\tau.
\end{align*}

\subsubsection{\textbf{The $C^\alpha$ estimate for $W_{G,x}$}} We consider the term
\begin{align}
W_{G,x}:=\int_0^t\int_{B_1}\Gamma_x(x,t;\xi.\tau)G(\xi,\tau)\,d\xi\,d\tau.
\end{align}
% \begin{align}
%     |L_G|\leq I_G+II_G
% \end{align}
Using the u-substitution $u:=\frac{x_2-\xi}{\sqrt{k_U}\sqrt{t-\tau}},$ we write
\begin{align}\nonumber
    I_{G,x}&\lesssim \|G\|_{L^\infty(\Omega_T)}\int_0^t\int_{B_1}\frac{|u|^{1-\alpha}|t-\tau|^{1-\frac{\alpha}{2}}}{{|t-\tau|^{3/2}}}\,e^{-|u|^2}\,du\,d\tau\\\label{c1}
    &\leq C\, \|G\|_{L^\infty(\Omega_T)}\int_0^t\frac{1}{|t-\tau|^{1/2+\alpha/2}}\,d\tau=C\, \|G\|_{L^\infty(\Omega_T)}\,t^{1/2-\alpha/2}.
\end{align}
Using the u-substitution $u:=\frac{x_1-\xi}{\sqrt{k_U}\sqrt{t-\tau}},$ we write
\begin{align}\nonumber
    II_{G,x}&\lesssim \|G\|_{L^\infty(\Omega_T)}\int_0^t\int_{B_1}\frac{|u|}{{|t-\tau|^{1-\epsilon}}}\,e^{-|u|^2}\,du\,d\tau\\\label{c2}
    &\leq C \, \|G\|_{L^\infty(\Omega_T)} t^{\epsilon}.
\end{align}
Similarly, 
\begin{align}\nonumber
    III_{G,x}&\lesssim \|G\|_{L^\infty(\Omega_T)}\int_0^t\int_{B_1}\frac{|u|^{\frac{3}{2}-\epsilon}}{{|t-\tau|^{\frac{3}{4}-\frac{\epsilon}{2}}}}\,e^{-|u|^2}\,du\,d\tau\\\label{c3}
    &\leq C \, \|G\|_{L^\infty(\Omega_T)} t^{\frac{1}{4}+\frac{\epsilon}{2}},
\end{align}
and
\begin{align}\nonumber
    IV_{G,x}&\lesssim \,\|G\|_{L^\infty(\Omega_T)}\int_0^t\int_{B_1}\bigg[\int_0^s\frac{d\tau}{(s-\tau)^{3/4}}\bigg]\|H_x(x-y,t-s)\|_{C^\alpha_x}\,dy\,ds\\\nonumber
    &\leq\|G\|_{L^\infty(\Omega_T)}\,t^{1/4}\int_0^t\int_{B_1}\frac{\big|H_x(x_2-y,t-s)-H_x(x_1-y,t-s)\big|}{|x_2-x_1|^\alpha}\,dy\,ds\\
    &\leq C\,\|G\|_{L^\infty(\Omega_T)}\,t^{1/4}\,\max\{t^{\frac{1}{2}-\alpha/2},t^\epsilon,t^{\frac{1}{4}+\frac{\epsilon}{2}}\}.
\end{align}

\subsubsection{\textbf{The $C^\alpha$ estimate for $W_{f,x}$}} We consider the term
\begin{align}
    W_{f,x}:=\int_0^t\int_\R\Gamma_x(x,t;\xi.\tau)f(\xi,\tau)\,d\xi\,d\tau.
\end{align}
Recall,
\begin{align*}
    I_{f,x}&=\int_0^t\int_{B_1}\frac{|x_2-\xi|^{1-\alpha}}{|t-\tau|^{3/2}}e^{-\frac{(x_2-\xi)^2}{4k_U(t-\tau)}}|f(\xi,\tau)|\,d\xi\,d\tau,\\
    II_{f,x}&=\int_0^t\int_{B_1}\frac{|x_1-\xi|}{|t-\tau|^{2-\epsilon}}e^{-\frac{(x_1-\xi)^2}{4k_U(t-\tau)}}|f(\xi,\tau)|\,d\xi\,d\tau,\\
    III_{f,x}&=\int_0^t\int_{B_1}\frac{|x_1-\xi|^{\frac{3}{2}-\epsilon}}{|t-\tau|^{2-\epsilon}}e^{-\frac{(x_1-\xi)^2}{4k_U(t-\tau)}}|f(\xi,\tau)|\,d\xi\,d\tau,
\end{align*}
and
\begin{align*}
    IV_{f,x}=\int_0^t\int_{B_1}\bigg[\int_0^s\int_\R\frac{const}{(s-\tau)^{5/4}}e^{\frac{-d(y-\xi)^2}{4(s-\tau)}}\,|f(\xi,\tau)|\,d\xi\,d\tau\bigg]\|H_x(x-y,t-s)\|_{C^\alpha_x}\,dy\,ds.
\end{align*}
We start with $I_{f,x},$
\begin{align}\nonumber
    I_{f,x}&\leq \bigg(\int_0^t\int_{B_1}\frac{|x_2-\xi|^{2-2\alpha}}{|t-\tau|^{3-2r}}e^{-\frac{(x_2-\xi)^2}{2k_U(t-\tau)}}\,d\xi\,d\tau\bigg)^{1/2}\bigg(\int_0^t\int_{B_1}\frac{1}{|t-\tau|^{2r}}|f(\xi,\tau)|^2\,d\xi\,d\tau\bigg)^{1/2}\\\nonumber
    &\lesssim\bigg(\int_0^t\int_{B_1}\frac{|u|^{2-2\alpha}|t-\tau|^{3/2-\alpha}}{|t-\tau|^{3-2r}}e^{-u^2}\,du\,d\tau\bigg)^{1/2}\bigg(\int_0^t\int_{B_1}\frac{1}{|t-\tau|^{2r}}|f(\xi,\tau)|^2\,d\xi\,d\tau\bigg)^{1/2}\\\nonumber
    &\lesssim \|f\|_{L^\infty((0,T),L^2(\R))} \bigg(\int_0^t\frac{1}{|t-\tau|^{\frac{3}{2}-2r+\alpha}}\,d\tau\bigg)^{1/2}\bigg(\int_0^t\int_{B_1}\frac{1}{|t-\tau|^{2r}}\,d\xi\,d\tau\bigg)^{1/2}\\
    &\lesssim \|f\|_{L^\infty((0,T),L^2(\R))} t^{-\frac{1}{4}+r-\frac{\alpha}{2}}\,t^{\frac{1}{2}-r}=t^{\frac{1}{4}-\frac{\alpha}{2}}\,\|f\|_{L^\infty((0,T),L^2(\R))} ,
\end{align}
for $\alpha+\frac{1}{4}<r<\frac{1}{2}.$ Similarly,
\begin{align}\nonumber
    II_{f,x}&\leq \bigg(\int_0^t\int_{B_1}\frac{|x_1-\xi|^{2}}{|t-\tau|^{4-2\epsilon-2r}}e^{-\frac{(x_1-\xi)^2}{2k_U(t-\tau)}}\,d\xi\,d\tau\bigg)^{1/2}\bigg(\int_0^t\int_{B_1}\frac{1}{|t-\tau|^{2r}}|f(\xi,\tau)|^2\,d\xi\,d\tau\bigg)^{1/2}\\\nonumber
    &\lesssim\bigg(\int_0^t\int_{B_1}\frac{|u|^{2}}{|t-\tau|^{\frac{5}{2}-2\epsilon-2r}}e^{-u^2}\,du\,d\tau\bigg)^{1/2}\bigg(\int_0^t\int_{B_1}\frac{1}{|t-\tau|^{2r}}|f(\xi,\tau)|^2\,d\xi\,d\tau\bigg)^{1/2}\\\nonumber
    &\lesssim \|f\|_{L^\infty((0,T),L^2(\R))} \bigg(\int_0^t\frac{1}{|t-\tau|^{\frac{5}{2}-2\epsilon-2r}}\,d\tau\bigg)^{1/2}\bigg(\int_0^t\int_{B_1}\frac{1}{|t-\tau|^{2r}}\,d\xi\,d\tau\bigg)^{1/2}\\
    &\lesssim \|f\|_{L^\infty((0,T),L^2(\R))} t^{-\frac{3}{4}+\epsilon+r}\,t^{\frac{1}{2}-r}= t^{-\frac{1}{4}+\epsilon}\,\|f\|_{L^\infty((0,T),L^2(\R))},
\end{align}
for $\frac{3}{4}-\epsilon<r<\frac{1}{2}.$
Note here we  need $\epsilon>1/4.$ This with the previous condition \eqref{eps-range} on $\epsilon$ we get \[\frac{1}{4}<\epsilon<\frac{1}{2}-\alpha.\]
By similar estimates, we have
\begin{align}\nonumber
    III_{f,x}&\leq \bigg(\int_0^t\int_{B_1}\frac{|x_1-\xi|^{3-2\epsilon}}{|t-\tau|^{4-2\epsilon-2r}}e^{-\frac{(x_1-\xi)^2}{2k_U(t-\tau)}}\,d\xi\,d\tau\bigg)^{1/2}\bigg(\int_0^t\int_{B_1}\frac{1}{|t-\tau|^{2r}}|f(\xi,\tau)|^2\,d\xi\,d\tau\bigg)^{1/2}\\\nonumber
    &\lesssim\bigg(\int_0^t\int_{B_1}\frac{|u|^{3-2\epsilon}}{|t-\tau|^{2-\epsilon-2r}}e^{-u^2}\,du\,d\tau\bigg)^{1/2}\bigg(\int_0^t\int_{B_1}\frac{1}{|t-\tau|^{2r}}|f(\xi,\tau)|^2\,d\xi\,d\tau\bigg)^{1/2}\\\nonumber
    &\lesssim \|f\|_{L^\infty((0,T),L^2(\R))} \bigg(\int_0^t\frac{1}{|t-\tau|^{2-\epsilon-2r}}\,d\tau\bigg)^{1/2}\bigg(\int_0^t\int_{B_1}\frac{1}{|t-\tau|^{2r}}\,d\xi\,d\tau\bigg)^{1/2}\\
    &\lesssim \|f\|_{L^\infty((0,T),L^2(\R))} t^{-\frac{1}{2}+\frac{\epsilon}{2}+r}\,t^{\frac{1}{2}-r}= t^{\epsilon/2}\,\|f\|_{L^\infty((0,T),L^2(\R))},
\end{align}
for $\frac{1}{2}-\frac{\epsilon}{2}<r<\frac{1}{2}.$
For the lower order term
\begin{align}\nonumber
    IV_{f,x}=&\int_0^t\int_\R\int_\tau^t\int_{B_1}\|H_x(x-y,t-s)\|_{C^\alpha_x}|\Phi(y,s;\xi,\tau)|\,dy\,ds\,|f(\xi,\tau)|\,d\xi\,d\tau\\
    &\leq \int_0^t\int_\R\int_\tau^t\int_{B_1}\|H_x(x-y,t-s)\|_{C^\alpha_x}\frac{const}{(s-\tau)^{5/4}}e^{\frac{-d(y-\xi)^2}{4(s-\tau)}}
\,dy\,ds\,|f(\xi,\tau)|\,d\xi\,d\tau.
\end{align}
Making the change of variables: $y= u+\xi$ and $s= v+\tau$ we obtain
\begin{align}
   =\int_0^t\int_\R\int_0^{t-\tau}\int_{\tilde B_1}\|H_x(x-u-\xi,t-v-\tau)\|_{C^\alpha_x}\frac{const}{(v)^{5/4}}e^{\frac{-d(u)^2}{4(v)}}|f(\xi,\tau)|
\,dV
\end{align}
where $dV=du\,dv\,\,d\xi\,d\tau.$ Similarly, we denote \[\|H_x(x-u-\xi,t-s)\|_{C^\alpha_x}:=\frac{\big|H_x(x_2-u-\xi,t-v-\tau)-H_x(x_1-u-\xi,t-v-\tau)\big|}{|x_2-x_1|^\alpha}.\]
Interchanging the integrals we get
%Switching $dy\,ds$ and $d\xi\,d\tau$ we have
\begin{align}\nonumber
   =&\int_0^t\int_\R\bigg[\int_0^{t-v}\int_\R \|H_x(x-u-\xi,t-s)\|_{C^\alpha_x}\,|f(\xi,\tau)|\,d\xi\,d\tau\bigg]\frac{const}{(v)^{5/4}}e^{-d\frac{u^2}{4v}}
\,du\,dv.\\\nonumber
&\leq \max\{t^{\frac{1}{4}-\frac{\alpha}{2}},t^{-\frac{1}{4}+\epsilon},t^{\epsilon/2}\}\|f\|_{L^\infty((0,T),L^2(\R))}\int_0^t\int_\R\frac{const}{(v)^{5/4}}e^{-d\frac{u^2}{4v}}
\,du\,dv\\
&\lesssim t^{3/4}\,\max\{t^{\frac{1}{4}-\frac{\alpha}{2}},t^{-\frac{1}{4}+\epsilon},t^{\epsilon/2}\}\|f\|_{L^\infty((0,T),L^2(\R))}.
\end{align}

Altogether, we obtain that $W_{G,x},W_{f,x}\in C^\alpha_x(\overline\Omega_T)$ for $\alpha\in(0,1/2).$ Moreover, due to the specific scaling of space and time in the heat kernel, that is \[(H)_t=(H)_{xx},\] we have that
\[W_{G,x},W_{f,x}\in C^\alpha_t(\overline\Omega_T),\] for $\alpha\in(0,1/4).$

\subsection{\textbf{The initial data}} We will show

First, note that $W(x,t)$ satisfies, in a classical sense, the Cauchy problem
\begin{align}\label{wsolinit}
    \begin{split}
     w_t-k(x,t)w_{xx}+\gamma w=0,\\
     w(x,0)=w_0(x)\in H^1(\R).   
    \end{split}
\end{align}
The calculation to show $W(x,t)\to w_0(x)$ is similar to the calculation in \cite{Fri} (Chapter 1, Theorem 11).

To show that $W_x(x,t)\to w'_{0}(x)$ we use standard second order parabolic equations theory. Denote $w_x:=u,$ using \eqref{wsolinit} we write the equation satisfied by $u$
\begin{align}\label{usolinit}
\begin{split}
    u_t-\big(k(&x,t)u_x\big)_x+\gamma u=0\\
    u(x,0)&=u_0(x)\in L^2(\R).
\end{split} 
\end{align}
It is known that the solution $u$ to \eqref{usolinit} belongs to $C([0,T],L^2(\R))$ and $u(x,0)= u_0(x)$ \cite{Evans}. Hence, by definition of $u$ we obtain $W_x(x,0)=w'_0(x).$

Finally, we emphasize that the regularity of the initial data $w_0(x)$ assumed gives that $W(x,t)$ and $W_x(x,t)$ are also H\"older continuous with the same exponent $\alpha$ away from $t=0.$ The calculations are quite standard and similar to the proof above so we omit this part.\\

In the next section, we show the existence of weak solution as a limit with respect to a mollifier's parameter. We will make use of the $L^\infty$ estimates of $w$ and $w_x.$

% Finally, it remains to show that $II\to 0$ as $t\to 0^+.$
% \begin{align*}
% &\bigg|\int_\R\int_0^t\int_\mathbb{R}H_x^{y,s}(x-y,t-s)\Phi(y,s;\xi,0)\,dy\,ds\,v_0(\xi)\,d\xi\bigg|\\
% &\leq \|v_0\|_{L^\infty}\int_\R\int_0^t\int_\R \frac{|x-y|}{(t-s)^{3/2}}e^{-g_L\frac{(x-y)^2}{4t}} \frac{1}{s^{5/4}}e^{-d\frac{(y-\xi)^2}{4s}}\,dy\,ds\,d\xi\\
% &=C\,\|v_0\|_{L^\infty}\int_0^t\int_\R \frac{|x-y|}{(t-s)^{3/2}}e^{-g_L\frac{(x-y)^2}{4t}} \frac{1}{s^{3/4}}\,dy\,ds\\
% \leq& C\,\|v_0\|_{L^\infty}\int_0^t \frac{1}{(t-s)^{1/2}} \frac{1}{s^{3/4}}\,ds\to 0,
% \end{align*}
% as $t\to 0^+.$ This completes the proof.

\section{\textbf{Global existence of weak solution}}

In this section we use standard results regarding mollifying functions with specific properties. Detailed proofs can be found in the appendix of \cite{Evans}.
Recall \eqref{source} and Remark \ref{G}, $$f(x,t)\in L^2(\Omega_T) \quad\text{and}\quad G(x,t)=\int_{-\infty}^x f^2(z,t)\,dz\in L^\infty(\Omega_T).$$ Denote the mollification of $f$ by $f^\ve.$ Therefore, the mollification of $G$ is 
\[G^\ve=\int_{-\infty}^x{f^\ve}^2\,dz.\]
Since both $f$ and $G$ are $L^2_{loc},$
We know that
$$f^\ve\to f \quad \text{and}\quad G^\ve\to G \mbox{ in } L^2_{loc} \mbox{ as } \ve\to 0.$$
% We also have $G(x,t),$ Here we write the given form of $G,$ as in Remark \ref{G} and \eqref{FGf},
% \[G(x,t)=\int_{-\infty}^xf^2(z,t)\,dz.\] We denote
% \[G^\ve(x,t):=\int_{-\infty}^x{f^\ve}^2(z,t)\,dz.\]
% Since $f^2\in L^1(\R),$
% we have that for a fixed $t,$
% \[{f^\ve}^2\to f^2\] in $L^1(\R).$ This means
% for any fixed $t,$
% \[G^\ve(x,t)\to G(x,t)\] uniformly in $x.$
% Taking the sup over $x\in\R$ and $t\in[0,T]$ we obtain
% \[\|G^\ve(x,t)-G(x,t)\|_{L^\infty([0,T],\R)}\to 0.\]
Following the discussion in \S 1.2, the Cauchy problem
\begin{align}\label{chye}
  w^\ve_t-k(x,t)w^\ve_{xx}+\gamma w^\ve&=G^\ve(x,t)+f^\ve(x,t),\\\label{inite}
w(x,0)&=w_0,  
\end{align}
has a classical solution given by
\begin{align}
w^\ve(x,t)=\int_\R\Gamma(x,t,\xi,0)w_0(\xi)\,d\xi+\int_0^t\int_\R\Gamma(x,t;\xi.\tau)[G^\ve(\xi,\tau)+f^\ve(\xi,\tau)]\,d\xi\,d\tau.
\end{align}
Therefore, the following identity holds
\begin{align}
        \int_0^T\int_\R-\phi_t w^\ve +(k\phi)_x w^{\ve}_x+\gamma \phi w^\ve\,dx\,dt=\int_0^T\int_\R(f^\ve+G^\ve)\phi\,dx\,dt,
    \end{align}
    for any $\phi\in C_c^1(\Omega_T).$ Using the $L^\infty$ estimates of $w,w_x$ in Section 2, we take $\ve\to 0$ to obtain
\begin{align}
        \int_0^T\int_\R-\phi_t w+(k\phi)_x w_x+\gamma \phi w\,dx\,dt=\int_0^T\int_\R(f+G)\phi\,dx\,dt,
    \end{align}
   for any $\phi\in C_c^1(\Omega_T).$

% \begin{remark}
% Note that the $L^2$ estimate of the solution is not needed for the existence result in this paper. However, such property is natural for the spatial derivative of the solution. More importantly, the $L^2$ decay is important for the existence of solution for some non-linear versions of a similar equation.
% \end{remark}

\section{\textbf{Proof of Proposition \ref{prop}}}

In this section, we show the $L^2$ estimates for $ W_{G,x}, W_{f,x}$ and $W_{f}.$

\subsection{\textbf{The $L^2$ estimate for $W_{G,x}$}}

We treat the term
\begin{align*}
W_{G,x}:=\int_0^t\int_\R\Gamma_x(x,t;\xi.\tau)G(\xi,\tau)\,d\xi\,d\tau.
\end{align*}
Recall,
\begin{align*}
W_G:=\int_0^t\int_{\R}\Gamma(x,t,\xi,\tau)G(\xi,\tau)\,d\xi\,d\tau
\end{align*}

Moreover, by the discussion in {Section} 1 we define $$W^\ve_G:=\int_0^t\int_{\R}\Gamma(x,t,\xi,\tau)G^\ve(\xi,\tau)\,d\xi\,d\tau$$ where $G^\ve$ is the mollification of $G.$ We have the following weak formulation satisfied
\begin{align}
  \int\int W^\ve_G\phi_t-(\phi\,k)_x\,W^\ve_{G,x}-\gamma_1 W^\ve_G\,\phi\,dx\,dt=-\int\int G^\ve\phi\,dx\,dt.  
\end{align}
for any $\phi\in C^1_0(\Omega_T).$
In fact, $W^\ve_G$ is a classical solution to the equation corresponding to the weak formulation. Choose $\phi=\eta_x,$
\begin{align}
  \int\int -W^\ve_{G,x}\eta_t+\eta_x\,k\,W^\ve_{G,xx}+\gamma_1 W^\ve_{G,x}\,\eta\,dx\,dt=\int\int G^\ve_{x}\eta\,dx\,dt.  
\end{align}
It is easy to see that $W^\ve_{G,x}$ satisfies
\begin{align}\label{Bve}
    (W^\ve_{G,x})_t-\big(k(x,t)(W^\ve_{G,x})_x\big)_x+\gamma_1 W^\ve_{G,x}=G^\ve_x
\end{align}
in a classical sense. Now, for short we denote $W^\ve_{G,x}$ by $B^\ve_x.$ Multiplying \eqref{Bve} by $B^\ve_x$ and integrating,
\begin{align}
   \frac{1}{2}\frac{d}{dt}\int_\R (B^\ve_x)^2\,dx+\int_\R k(x,t)(B^\ve_{xx})^2+\gamma_1(B^\ve_x)^2\,dx=\int_\R G^\ve_xB^\ve_x\,dx.
\end{align}
Integrating in $t,$
\begin{align}\label{B^ve}
   \frac{1}{2}\int_\R (B^\ve_x)^2\,dx+\int_0^t\int_\R k(x,s)(B^\ve_{xx})^2+\gamma_1(B^\ve_x)^2\,dx\,ds=\int_0^t\int_\R G^\ve_xB^\ve_x\,dx\,ds.
\end{align}
 Now the given form of $G,$ that is $G_x=f^2,$ implies that $G^\ve_x\to G_x$ in $L^1.$ and we also have, from the previous sections, $B^\ve_x\to B_x$ in $L^\infty.$ Together, we obtain for small $\ve,$
 \[\int_0^t\int_\R G^\ve_xB^\ve_x\,dx\leq \frac{1}{2} \int_0^t\int_\R G_xB_x\,dx\leq\frac{1}{2}\,T\, \|B_x\|_{L^\infty([0,T],\R)}\|G_x\|_{L^\infty([0,T],L^1(\R))}.\] Hence, by \eqref{B^ve}  and \eqref{infityM_Gx} we have
 \begin{align}
    \|B^\ve_x\|_{L^2([0,T],\R)}\leq \frac{1}{2}T^{3/2} \|G\|_{L^\infty([0,T],\R)}\|G_x\|_{L^\infty([0,T],L^1(\R))}.
 \end{align}
  Since $B^\ve_x\to B_x$ in $L^\infty,$ there exists a subsequence, still denoted by $B^\ve_x,$ converges weakly to $B_x$ in $L^2.$ Hence, $W_{G,x}\in L^2$ and 
  $$\|W_{G,x}\|_{L^2([0,t],\R)}\leq \frac{1}{2}T^{3/2} \|G\|_{L^\infty([0,T],\R)}\|G_x\|_{L^\infty([0,T],L^1(\R))}.$$

\subsection{\textbf{The $L^2$ estimate for $W_{f,x}$ and $W_f$}} We consider the term
\begin{align}
    W_{f,x}:=\int_0^t\int_\R\Gamma_x(x,t;\xi.\tau)f(\xi,\tau)\,d\xi\,d\tau,
\end{align}
and use Young's inequality with the help of \eqref{Gamma_xest} we get,
\begin{align}
  \|W_{f,x}\|_{L^2}&\leq \bigg\|\int_0^t\int_\R\frac{cont}{{t-\tau}}e^{-\frac{d(x-\xi)^2}{4(t-\tau)}}|f(\xi,\tau)|\,d\xi\,d\tau\bigg\|_{L^2}\\
&\approx\|\tilde H\ast f\|_{L^2}\leq \|\tilde H\|_{L^{1}}\|f\|_{L^2}\leq T^{1/2}\|f\|_{L^\infty((0,T),L^2(\R))}, 
\end{align}
where $\tilde H=\frac{1}{\sqrt{t}}H.$ Similarly, one can show the $L^2$ of $W_f.$\\

\textbf{Acknowledgement}: I would like to thank my advisors Prof. Geng Chen and Prof. Weishi Liu for their valuable comments and guidance. This paper is motivated by a project that we have conducted recently. I would also like to thank Prof. Lihe Wang for referring me to some of the relevant previous results in the literature.

% \section{\textbf{Conclusion}}
%  To conclude, we should emphasize that the solution \eqref{sol} loses the H\"older continuity as $t\to 0^+.$ In fact, the solution $u$ approaches $u_0$ point-wise only in a continuous manner, that is
%  \[u(x,t)\in \big[(L^\infty\cap C)(\overline \Omega_T)\big]\cap \big[(L^\infty\cap L^2\cap C^\alpha)(\Omega_T)\big].\]

\end{document}